\documentclass[11pt]{article}
\usepackage[margin=1in]{geometry}
\usepackage{graphicx,amssymb,amsmath,amsthm,xcolor, enumitem}
\usepackage{cite}
\usepackage{url}
\usepackage{hyperref}

\usepackage{aliascnt}
\theoremstyle{theorem}
\newtheorem{theorem}{Theorem}[section]
\newaliascnt{lemma}{theorem}
\newtheorem{lemma}[lemma]{Lemma}
\aliascntresetthe{lemma}

\newaliascnt{corollary}{theorem}

\aliascntresetthe{corollary}

\newaliascnt{claim}{theorem}
\newtheorem{claim}[claim]{Claim}
\aliascntresetthe{claim}

\title{Face Flips in Origami Tessellations}

\author{Hugo A. Akitaya\thanks{Tufts University, \protect\url{hugo.alves_akitaya@tufts.edu}} \and Vida Dujmovi{\'c}\thanks{University of Ottawa, \protect\url{vida.dujmovic@uottawa.ca}} \and David Eppstein\thanks{University of California, Irvine, \protect\url{eppstein@uci.edu}} \and Thomas C. Hull\thanks{Western New England University,\protect\url{thull@wne.edu}}  \and Kshitij Jain\thanks{Borealis AI, \protect\url{k22jain@uwaterloo.ca}} \and Anna Lubiw\thanks{University of Waterloo, \protect\url{alubiw@uwaterloo.ca}} }

\date{}

\begin{document}

%\newtheorem{thm}{Theorem}[section]
%\newtheorem{cor}[thm]{Corollary}
%\newtheorem{lem}[thm]{Lemma}
%\newtheorem{conj}[thm]{Conjecture}
%\newtheorem{prop}[thm]{Proposition}
%
%\theoremstyle{definition}
%\newtheorem{defin}[thm]{Definition}
%
%\newcounter{prob-counter}
%\theoremstyle{definition}
%\newtheorem{prob}[prob-counter]{Problem}
%
%\theoremstyle{definition}
%\newtheorem{example}[thm]{Example}

\newcommand{\R}{\mathbb{R}}

\maketitle

\begin{abstract}
Given a flat-foldable origami crease pattern $G=(V,E)$ (a straight-line drawing of a planar graph on a region of the plane) with a mountain-valley (MV) assignment $\mu:E\to\{-1,1\}$ indicating which creases in $E$ bend convexly (mountain) or concavely (valley), we may \emph{flip} a face $F$ of $G$ to create a new MV assignment $\mu_F$ which equals $\mu$ except for all creases $e$ bordering $F$, where we have $\mu_F(e)=-\mu(e)$.  In this paper we explore the configuration space of face flips for a variety of crease patterns $G$ that are tilings of the plane, proving examples where $\mu_F$ results in a MV assignment that is either never, sometimes, or always flat-foldable for various choices of $F$.  We also consider the problem of finding, given two foldable MV assignments $\mu_1$ and $\mu_2$ of a given crease pattern $G$, a minimal sequence of face flips to turn $\mu_1$ into $\mu_2$.  We find polynomial-time algorithms for this in the cases where $G$ is either a square grid or the Miura-ori, and show that this problem is NP-hard in the case where $G$ is the triangle lattice.
\end{abstract}

\section{Introduction}

An {\em origami crease pattern} $(G,P)$ is a straight-line drawing of a planar graph $G=(V,E)$ on a region $P$ of $\R^2$, where we allow for the case $P=\R^2$ and $G$ is an infinite graph.  A {\em flat origami} is a function $f:P\to\R^2$ from an origami crease pattern $(G,P)$ to the plane that is continuous, an isometry on each face of $G$, and non-differentiable on all the edges and vertices of $G$.  The combinatorics of flat origamis have been studied somewhat extensively (see \cite{Hull2} for a survey), but many open questions remain.  

For example, since flat origamis aim to model the folded state of paper when folded completely flat, we may record the state of each crease segment with a function $\mu:E\to\{-1,1\}$, where $\mu(e)=-1$ means that  the crease $e$ is a {\em valley crease} (meaning it bends the paper in a concave direction) and $\mu(e)=1$ means that $e$ is a {\em mountain crease} (so it bends in a convex direction).  We refer to $\mu$ as a {\em mountain-valley (MV) assignment} on $G$, and MV assignments that do not force the paper to self-intersect when physically folded are called {\em valid}.   Characterizing and enumerating valid MV assignments are open problems.  Even restricting ourselves to MV assignments that are {\em locally valid}, meaning that we only require that each vertex not force a self-intersection of physically-folded paper, is of interest, having been studies in statistical mechanics \cite{Assis} and applied to polymer membrane folding \cite{DiF}.  As another example, counting locally-valid MV assignments for the family of crease patterns known as the {\em Miura-ori} has been shown to be equivalent to counting proper vertex colorings of grid graphs \cite{GHull}, which is an unsolved problem.  

A new combinatorial tool that shows promise for helping explore locally-valid MV assignments is the face flip.  If $F$ is a face in a flat-foldable crease pattern $(G,P)$ and we have a MV assignment $\mu$, then a {\em face flip of $F$ in $(G,P)$ under $\mu$} is a new MV assignment $\mu_F$ where $\mu_F=\mu$ for all edges in $G$ except those that border $F$, where we have $\mu_F=-\mu$.  That is, we ``flip'' the creases bordering $F$ from mountain to valley and vice-versa.  Face flips seem to have been first introduced by Kyle VanderWerf in \cite{Van14}, but they are otherwise unexplored.

In this paper, we examine the properties of face flips on flat origami crease patterns $(G,P)$  where $G$ is certain regular tilings of the plane.  Such flat origamis are also known as {\em origami tessellations}, and they are of central interest in applications and prior work on flat foldings \cite{Assis,DiF,Evans,GHull,Silverberg}.  

Specifically, after setting up background results in Section~\ref{sec1},  we will see in Section~\ref{sec2} families of quadrilateral crease patterns where any face flip on a MV assignment will preserve its local validity, another where only certain faces may be flipped, and yet another where no face flip will make a valid MV assignment.  In Section~\ref{sec3} we will prove that any locally-valid MV assignment of the Miura-ori crease pattern can be converted to any other via face flips, thus showing that the configuration space of locally-valid MV assignments of the Miura-ori is connected under face flips.  We also employ a height function to determine the minimum number of face flips needed to traverse this configuration space.  In Section~\ref{sec4} we consider the considerably more complicated case of origami tessellations whose crease pattern is the triangle lattice, showing that its configuration space is also connected.  However, we show that, unlike the quadrilateral cases, determining the minimum number of face flips needed to convert one locally-valid MV assignment of the triangle lattice to another is NP-hard using a reduction from minimum vertex cover with maximum degree three in a hexagonal grid.  

\section{Preliminaries}\label{sec1}

The most fundamental result of flat-foldability is {\em Kawasaki's Theorem}:

\begin{theorem}[Kawasaki]\label{Kawasaki}
Let $(G,P)$ be an origami crease pattern where $G$ has only one vertex $v$ in the interior of $P$ and all edges in $G$ are adjacent to $v$.  Let $\alpha_1,\ldots,\alpha_{k}$ be the sector angles, in order, between the consecutive edges around $v$.  Then there exists a flat origami function for $(G,P)$ if and only if $k=2n$ is even and
$$\alpha_1-\alpha_2+\alpha_3-\cdots-\alpha_{2n}=0.$$
\end{theorem}

See \cite{Hull2} for a proof.

One of the most basic requirements for a MV assignment to be valid is for there to be an appropriate number of mountains and valleys at each vertex.

\begin{theorem}[Maekawa]\label{Maekawa}
Let $v$ be a vertex in a flat-foldable crease pattern with a valid MV assignment $\mu$ and let $E$ be the set of crease edges adjacent to $v$.
%Let  $M$ (resp. $V$) denote the number of mountain (resp. valley) creases adjacent to $v$ under $\mu$.  Then $M-V=\pm 2$.
Then 
$$\sum_{e\in E}\mu(e) = \pm 2.$$
\end{theorem}

This is known as {\em Maekawa's Theorem}, although it is often written as $M-V=\pm2$ where $M$ and $V$ are the number of mountain and valley creases, respectively, at $v$.  A proof can be found in \cite{Hull2}; it relies on the fact that the cross section of a physically-flat-folded vertex will be a closed curve with turning number 1, giving us that  $\pi M-\pi V=\pm2\pi$.  

If the sector angles around a flat-foldable vertex are all equal, then by symmetry we can choose any of them to be mountains and valleys so long as $\sum\mu(e)=\pm 2$.  This implies the following:

\begin{theorem}\label{MaekawaNASC}
Let $v$ be a vertex in a flat-foldable crease pattern whose sector angles between consecutive creases are all equal.  Then a MV assignment $\mu$ will be valid at the vertex $v$ if and only if $\mu$ satisfies Maekawa's Theorem at $v$.
\end{theorem}

When the angles between consecutive creases around a flat-folded vertex are not all equal, then Maekawa is only a necessary condition.  Nonetheless, other constraints on valid MV assignments for such vertices can be deduced.  For example, if we have consecutive sector angles $\alpha_{i-1}, \alpha_i, \alpha_{i+1}$, with $\alpha_i$ between crease edges $e_i$ and $e_{i+1}$, at a vertex with a valid MV assignment $\mu$, where $\alpha_i$ is strictly smaller then both of the other angles, then we must have $\mu(e_i)=-\mu(e_{i+1})$.  This is because if $\mu(e_i)=\mu(e_{i+1})$ then the two sectors with angles $\alpha_{i-1}$ and $\alpha_{i+1}$ would be folded over, and more than cover, the sector with angle $\alpha_i$ on the same side of the paper, causing the sectors of paper with angles $\alpha_{i-1}$ and $\alpha_{i+1}$ to intersect each other.  This constraint, where $\mu(e_i)=-\mu(e_{i+1})$ is forced, is sometimes called the {\em Big-Little-Big Angle Lemma}.  This is actually a special case of the following, proved in \cite{Hull2}:

\begin{theorem}\label{genMaekawa}
Let $v$ be a vertex in a flat-foldable crease pattern with a valid MV assignment $\mu$, and suppose that  we have a local minimum of consecutive equal sector angles between the crease edges $e_i,\ldots,e_{i+k+1}$ at $v$.  That is, $\alpha_i=\alpha_{i+1}=\cdots =\alpha_{i+k}$ where $\alpha_{i-1}>\alpha_i$ and $\alpha_{i+k+1}>\alpha_i$.  Then
$$\sum_{j=i}^{i+k+1}\mu(e_j) = \left\{
\begin{array}{cl}
0 & \mbox{if $k$ is even,}\\
\pm 1 & \mbox{if $k$ is odd.}
\end{array}\right.$$
\end{theorem}

In the Introduction we defined face flips of a flat origami with a MV assignment. %One fact about face flipping that will be quite useful is the following.
We also say that two MV assignments $\mu_1$ and $\mu_2$ are {\em face-flippable} if there exists a sequence of faces in the crease pattern  whose flipping will turn $\mu_1$ into $\mu_2$, or vice-versa.  

We now formalize the concept of the configuration space of locally-valid MV assignments, which we do with a variation of the flip graph from discrete geometry.  Given a flat origami crease pattern $(G,P)$, define the {\em origami flip graph} to be the graph whose vertices are locally-valid MV assignments of $(G,P)$, and where two MV assignments $\mu_1$ and $\mu_2$ are adjacent in this graph if and only if  $\mu_1$ is face-flippable to $\mu_2$ by flipping exactly one face.  We then say that the configuration space of MV assignments for $(G,P)$ is connected if its origami flip graph is a connected graph.

\section{Square and kite tessellations}\label{sec2}

In this section we will analyze three families of quadrilateral-based origami tessellations to demonstrate different kinds of face flip behavior and different configuration spaces.  Specifically, we will see:

\begin{itemize}
\item {\em Square grid} crease patterns, where any face can be flipped and the configuration space of MV assignments is connected.
\item {\em Huffman grid} tessellations, where no face flips are possible and thus the configuration space is totally disconnected.
\item {\em Square twist} tessellations, where only half of the faces can be flipped but the configuration space is still connected.
\end{itemize}

\subsection{Square grid tessellations}

We first consider an $m\times n$ grid of squares, denoted $G_{m,n}$, as our crease pattern, where the region of paper $P$ will be an $m\times n$ rectangle.  This will have $(m-1)(n-1)$ vertices in the interior of $P$, each of which will have degree 4 and $90^\circ$ angles between the creases.  Thus a MV assignment $\mu$ for $G_{m,n}$ will be locally valid if and only if each vertex satisfies Maekawa's Theorem, i.e., has 3 mountains and 1 valley or vice-versa.  

\begin{theorem}\label{anygrid}
Given a square grid tessellation and a locally-valid MV assignment $\mu$, flipping any face will result in another locally-valid MV assignment.
\end{theorem}

\begin{proof}
Flipping a face $F$ in $G_{m,n}$ will affect $\mu$ for at most four interior vertices in $P$.  Consider one of them, $v$, and let $e_1,\ldots, e_4$ be the crease edges adjacent to $v$ where $e_1$ and $e_2$ border $F$.  If $\mu(e_1)=\mu(e_2)$ then $\sum_{i=1}^4 \mu(e_i)=-\sum_{i=1}^4\mu_F(e_i)$.  If $\mu(e_1)=-\mu(e_2)$ then $\mu_F(e_1)=-\mu_F(e_2)$ as well.  In both cases we have that Maekawa's Theorem is still satisfied at $v$ under $\mu_F$, and thus $\mu_F$ will be locally-valid.
\end{proof}

\begin{lemma}\label{gridlemma}
Let $\mu_1$ and $\mu_2$ be two locally-valid MV assignments of $G_{m,n}$ and let $v$ be an interior vertex of $G_{m,n}$.  Then among the four edges adjacent to $v$, $\mu_1$ and $\mu_2$ can agree on all four, only 2, or none of the edges, but not on three or one.
\end{lemma}

\begin{proof}
Let the edges at $v$ be $e_1,\ldots, e_4$, and suppose that $\mu_1$ and $\mu_2$ agree on only one or three of these edges.  Then in either case we have that $\prod_{i=1}^4 \mu_1(e_i)\mu_2(e_i) = -1$, since the disagreeing pairs of $\mu_1(e_i)$ and $\mu_2(e_i)$ will each contribute  $-1$ and the agreeing pairs will contribute  $1$ to the product.  However, this product also equals $\prod_{i=1}^4\mu_1(e_i) \prod_{i=1}^4\mu_2(e_i) = 1$ since $\prod_{i=1}^4\mu(e_i)=-1$ for all locally-valid MV assignments $\mu$.  This is a contradiction.
\end{proof}

\begin{figure}
\centerline{\includegraphics[width=\linewidth]{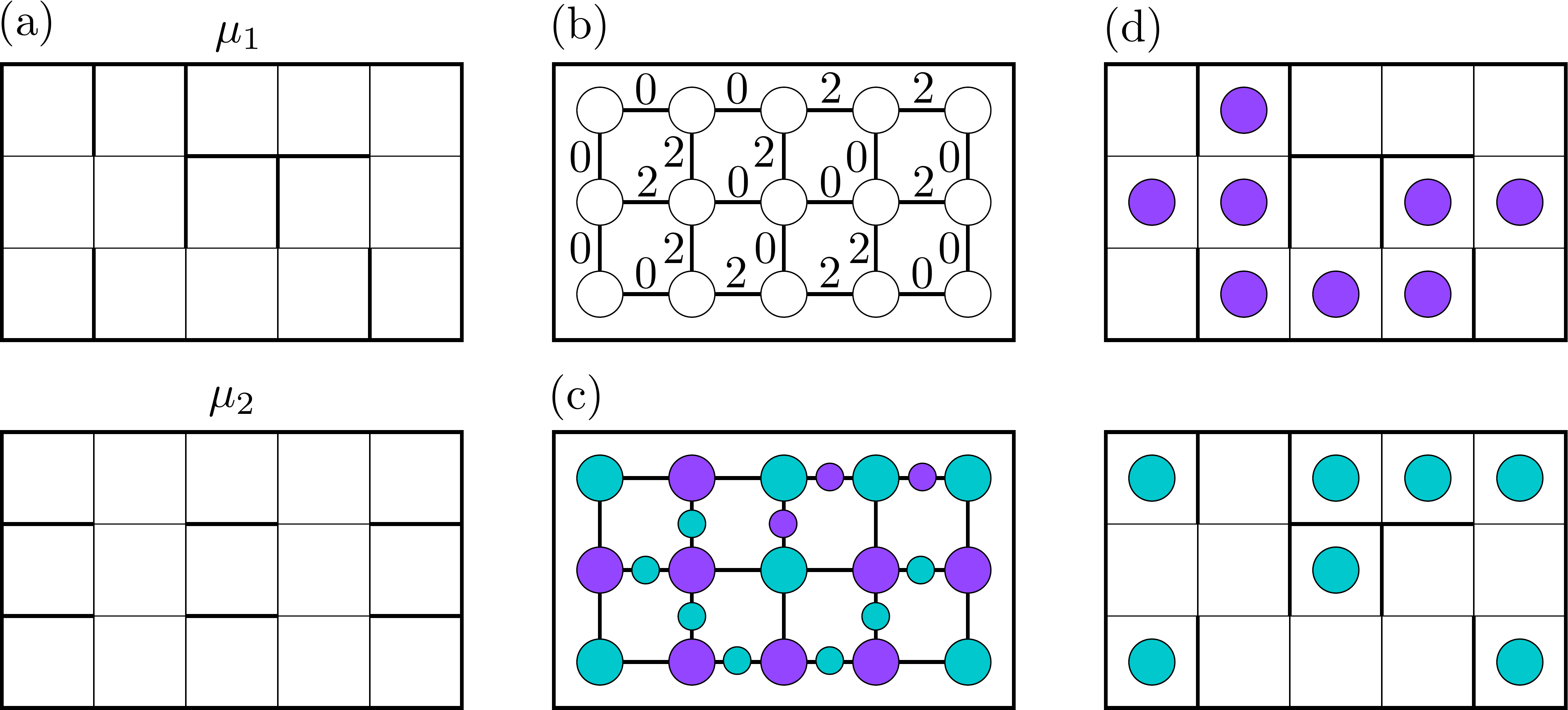}}
\caption{(a) Two locally-valid MV assignments $\mu_1$ and $\mu_2$ of $G_{3,5}$.  (b) The weight $w$ on the edges of $G^*(3,5)$. (c) The graph $\overline{G^*}_{3,5}$ with a 2-coloring. (d) The face flip sets.}\label{fig0}
\end{figure}

Our goal is to prove that the MV configuration space for $G_{m,n}$ is connected and to devise an algorithm to find the smallest number of face flips needed to flip between two given locally-valid MV assignments $\mu_1$ and $\mu_2$.  

To that end, let $\mu_1$ and $\mu_2$ be locally-valid MV assignments of $G_{m,n}$, such as those shown in Figure~\ref{fig0}(a).  We consider the internal planar dual graph, $G^*_{m,n}$ (that is, the dual of $G_{m,n}$ ignoring the external face).  For every edge $e$  of $G_{m,n}$ denote the corresponding edge in $G^*_{m,n}$ by $e^*$.  Assign a weight function $w$ to the edges in $G^*_{m,n}$ given by $w(e^*) = |\mu_1(e)+\mu_2(e)|$.  That is, $w(e^*)$ will equal 0 if $\mu_1(e)\not=\mu_2(e)$ and 2 if $\mu_1(e)=\mu_2(e)$.  See the example in Figure~\ref{fig0}(b).

Now create a new graph $\overline{G^*}_{m,n}$ made by taking $G^*_{m,n}$ and adding a vertex in the middle of every edge $e^*$ with $w(e^*)=2$.  See Figure~\ref{fig0}(c). 

\begin{lemma}\label{2colorgrid}
The graph $\overline{G^*}_{m,n}$ is properly 2-vertex colorable.
\end{lemma}

\begin{proof}
This follows from Lemma~\ref{gridlemma};  since the four edges $e_i$ adjacent to each internal vertex $v$ of $G_{m,n}$ have $\mu_1(e_i)=\mu_2(e_i)$ for only four, two, or none of the $e_i$, we have that each square face of $G^*_{m,n}$ will either remain a square, become a hexagon, or become an octagon in $\overline{G^*}_{m,n}$.  Thus $\overline{G^*}_{m,n}$ has only even cycles, which means it is properly 2-vertex colorable.
\end{proof}

Let $c:E(\overline{G^*}_{m,n})\to\{$purple, teal$\}$ be a proper 2-vertex coloring of $\overline{G^*}_{m,n}$.  Then $c$ will give us a (most likely not proper) 2-coloring of the vertices of $G^*_{m,n}$.

\begin{theorem}\label{gridsconnected}
Let $\mu_1$ and $\mu_2$ be two locally-valid MV assignments of $G_{m,n}$.  Then $\mu_1$ and $\mu_2$ are face-flippable, and this can be achieved by flipping the faces corresponding to all the purple (or all the teal) vertices in $G^*_{m,n}$ under the 2-coloring $c$ described above.
\end{theorem}

\begin{proof}
Suppose we start with the MV assignment $\mu_1$ and flip all the faces corresponding to the purple vertices in $G^*_{m,n}$.  Consider an edge $e$ of $G_{m,n}$ where $\mu_1(e)=\mu_2(e)$.  Then $w(e^*)=2$, and thus the faces in $G_{m,n}$ that border $e$ have corresponding vertices in $G^*_{m,n}$ that are both colored purple or teal.  This means that both of these faces were flipped or both were not flipped.  In either case, the edge $e$ remains with the same MV assignment after all the purple flips.

Now consider $e\in E(G_{m,n})$ where $\mu_1(e)\not=\mu_2(e)$. Then $w(e^*)=0$, so the vertices adjacent to $e^*$ in $G^*_{m,n}$ are colored differently under $c$, which means one of the faces bordering $e$ in $G_{m,n}$ is flipped and the other is not flipped.  This implies that the edge $e$ will change its MV assignment after all the purple flips.  We conclude that the MV assignment $\mu_1$ will turn into $\mu_2$ after flipping all the faces corresponding to the purple vertices in $G^*_{m,n}$, and the same argument works for the teal vertices.
\end{proof}

It is clear from the construction that the two face flip sets generated by the purple vertices and by the teal vertices in $G^*_{m,n}$ are minimal in that removing any face from them will not result in flipping from $\mu_1$ to $\mu_2$ or vice-versa.  Could there be some other minimal set of faces that flips from $\mu_1$ to $\mu_2$?  No, since every edge $e$ in $G_{m,n}$ with $\mu_1(e)\not=\mu_2(e)$ requires one of its adjacent faces to be in the flip set, which is exactly what the purple and teal sets achieve.  We conclude that the smallest number of face flips needed to flip from $\mu_1$ to $\mu_2$ is the smaller of the purple and the teal face flip sets.

For a few interesting examples, if $\mu_1=\mu_2$ then one of the color sets, say purple, will include all the faces of $G_{m,n}$ and the other, teal, will be the empty set.  Clearly the empty set is the smaller set, meaning that no flips are needed.  If $\mu_1(e)=-\mu_2(e)$ for all $e\in E(G_{m,n})$ then  the 2-coloring of $G^*_{m,n}$ will be a simple checkerboard coloring, and if $mn$ is even then the purple and teal flip sets will have equal size.

\subsection{Huffman grid tessellations}

A {\em Huffman grid} is a type of monohedral origami tessellation introduced by Huffman in \cite{Huff}.  (See also \cite{Evans}.)  The generating tile is a quadrilateral with two opposite corners having right angles, as in Figure~\ref{fig1}(a).  The other two interior angles are labeled $\alpha$ and $\pi-\alpha$, where we assume $\alpha<\pi/2$.  The tiling generated by this tile, shown in Figure~\ref{fig1}(b), has vertices that satisfy Kawaski's Theorem, and by the Big-Little-Big Lemma, the creases bordering an angle of $\alpha$ must have different MV parity.  The creases that do not border an $\alpha$ angle form zig-zag paths, which we call {\em short rows}; they are highlighted blue in Figure~\ref{fig1}(b).  Applying Maekawa's Theorem at each vertex implies that each short row must be either entirely mountain or entirely valley creases.  As seen in \cite{Evans}, folding large Huffman grids according to a locally-valid MV assignment will eventually cause the paper to curl up and self-intersect.

\begin{figure}
\centerline{\includegraphics[scale=.3]{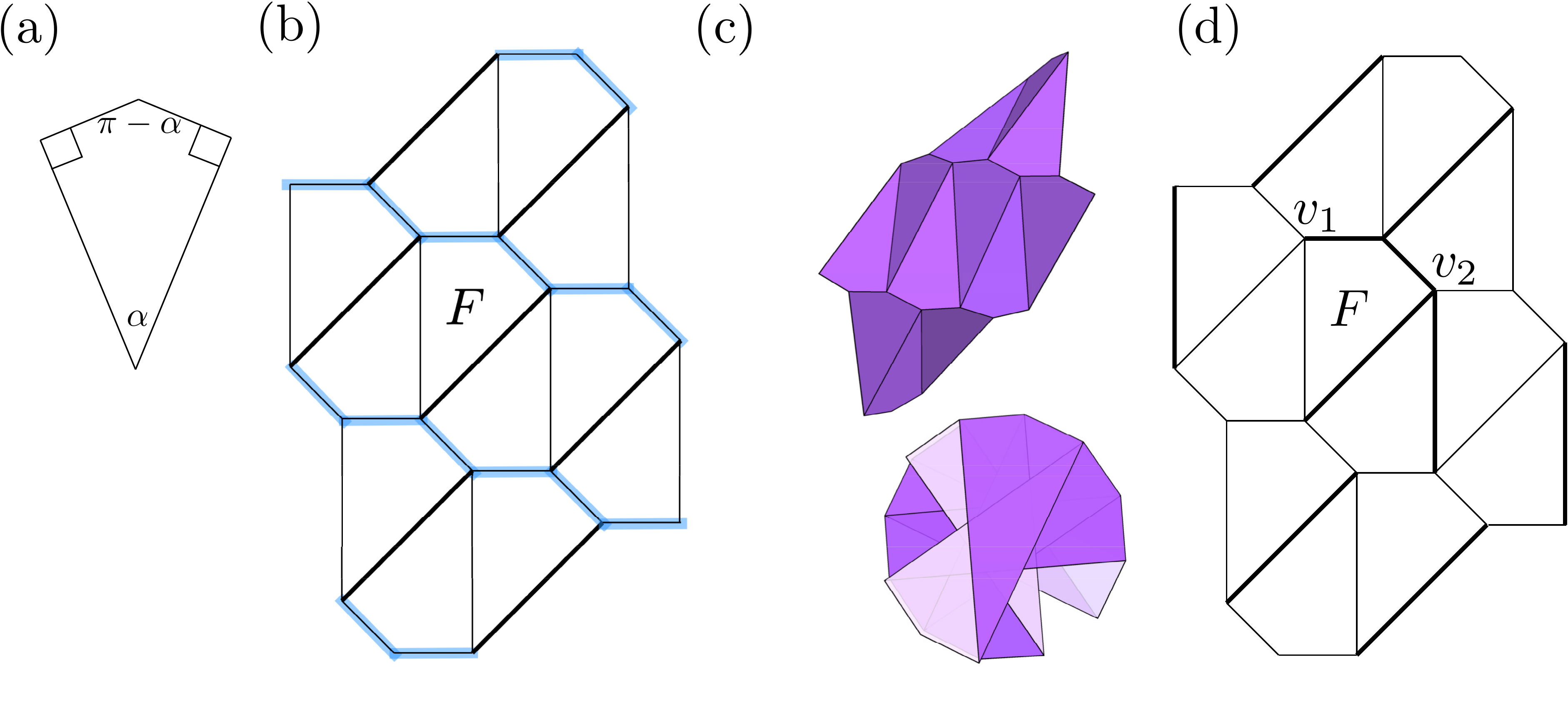}}
\caption[Caption]{(a) The kite tile for a Huffman grid.  (b) The resulting Huffman grid, with the short rows highlighted in blue.  (c) The crease pattern partially folded and folded flat.\footnotemark (d) The face $F$ flipped.}\label{fig1}
\end{figure}

\begin{theorem}\label{Huffmangrid}
Given the Huffman grid tessellation and a MV assignment, any face flip will generate a MV assignment that is not valid.
\end{theorem}

\footnotetext{The images in Figure~\ref{fig1}(c) were generated using R. J. Lang's {\em Tessellatica 11.1} Mathematica code \cite{TessLang}.}

\begin{proof}
Suppose we flip a face $F$ where we label $F$'s right angle corners $v_1$ and $v_2$.  Then before we flipped $F$, the two creases surrounding the angle $\alpha$ at $v_1$ had opposite MV parity.  After flipping $F$ these creases will have the same MV parity, which violates the Big-Little-Big Lemma at $v_1$, and the same thing will happen at $v_2$.  Thus if $\mu$ was our original locally-valid MV assignment, the flipped assignment $\mu_F$ will not be locally valid.
\end{proof}

Theorem~\ref{Huffmangrid} implies that there are no edges in the face flip graph for Huffman grid tessellations, and thus its configuration space is as disconnected as possible.

\subsection{Square twist tessellations}

Twist folds, which are collections of creases that, when folded flat, cause a polygon to rotate from the unfolded to the flat-folded state relative to the rest of the paper, are of interest for their geometric character \cite{Evans2} as well as their applications in origami mechanics \cite{Silverberg}, as well as their ability to be used in tessellations.  The classic {\em square twist}, where identical degree-4 vertices with sector angles $45^\circ$, $90^\circ$, $135^\circ$, and $90^\circ$ are used to twist a square, can tessellate in several different ways; a few examples are shown in Figure~\ref{fig2}.  In all square twist tessellations the $90^\circ$ angles will form square or restangle faces in the crease pattern tiling, while the  $45^\circ$ and $135^\circ$ angles will form parallelograms or trapezoids, as shown in Figure~\ref{fig2}(b).  

The Big-Little-Big Lemma implies that the  two creases bordering the $45^\circ$ angles in the vertices of a square twist tessellation must have different MV assignments.  This immediately gives us that if we were to face flip any one of the square or rectangle faces in a square twist tessellation then the Big-Little-Big Lemma would be violated at the $45^\circ$ angles that border the flipped face.  Thus any square or rectangle face cannot be flipped by itself.  In fact, the only way to flip a square or rectangle face and not violate the Big-Little-Big Lemma somewhere is to flip {\em all} of the square and rectangle faces, which is equivalent to flipping all the non-square/rectangle faces.

\begin{figure}
\centerline{\includegraphics[width=\linewidth]{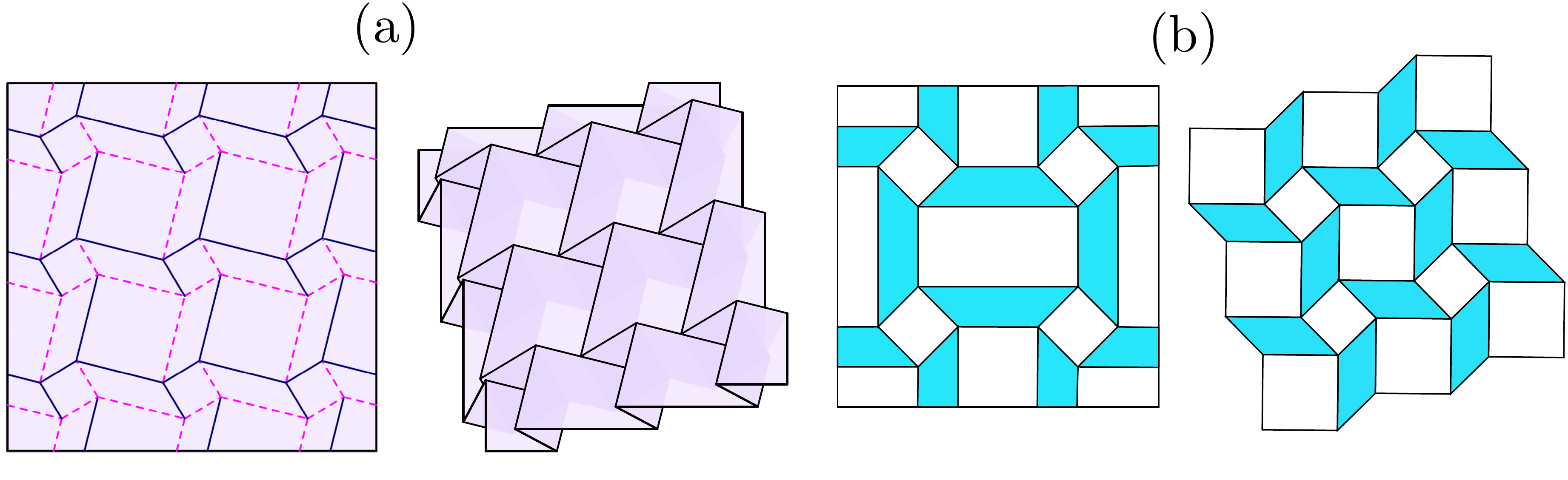}}
\caption[Caption]{(a) A square twist tessellation with MV assignment and the folded form.\footnotemark  (b) A square twist tessellation with the flippable faces colored.}\label{fig2}
\end{figure}

Therefore the only faces that are safe to flip (and preserve local validity of the MV assignment) in a square twist crease pattern are the parallelograms and trapezoids, shown in blue in Figure~\ref{fig2}(b).  Since all of these parallelograms and trapezoids are edge-disjoint, and these represent the only ways to modify a MV assignment of such patterns, we arrive at the following:

\begin{theorem}\label{squaretwist}
The MV assignment configuration space of a square twist tessellation crease pattern is connected, and the only faces that are flippable are the parallelograms and trapezoids.
\end{theorem}

\footnotetext{The images in Figure~\ref{fig2}(a) were generated using R. J. Lang's {\em Tessellatica 11.1} Mathematica code \cite{TessLang}.}

\section{The Miura-ori}\label{sec3}

The Miura-ori~\cite{GHull,Van14} is a folding pattern named after Koryo Miura, formed by a sequence of equally spaced parallel lines
crossed by zigzag paths that divide the paper into equal parallelograms. These parallelograms tile the plane by reflection across the parallel lines, and by translation in the direction parallel to the parallel lines. In its classical global flat-folded state, the folds of the Miura-ori alternate between mountain and valley folds along each of the parallel lines, with each zigzag path consisting entirely of mountain folds or entirely of valley folds (in alternation along the sequence of zigzag paths). However, the same folding pattern has many locally flat-folded states. In a locally flat-folded state of this folding pattern, each vertex must have three folds of one type (mountain or valley) and one fold of the opposite type. Additionally, if two of the folds of the same type come from the zigzag path through the vertex, the third fold of the same type must be on the remaining edge that is farthest (by angle) from these two edges.

The parallelograms of the Miura-ori have the same combinatorial structure (although different in their symmetry groups) as the square grid, and it will be helpful for us to use a combinatorial bijection between local flat foldings of the Miura-ori and 3-colorings of the vertices of a grid graph. To construct the bijection, we think of the grid graph as being the dual graph of the Miura-ori (with one vertex in each parallelogram) and the three colors as corresponding to the three integer values modulo three. We then follow a path through the parallelograms of the Miura-ori, starting from one corner continuing as long as possible in the direction parallel to the parallel fold lines of the pattern. Whenever this path reaches the opposite side of the folded paper, it makes a step across one of these parallel fold lines and then continues, as long as possible in the opposite direction parallel to this fold line. We choose the color corresponding to 0 mod 3 for the first parallelogram in the path. Then, when the path crosses a mountain fold, we add one mod 3, and when it crosses a valley fold, we subtract one mod 3. (See \autoref{fig:Miurabijection}.) It can be shown that this method of coloring the parallelograms produces a valid 3-coloring for every local flat-folding of the Miura-ori and, conversely, that every 3-coloring comes from a local flat-folding in this way~\cite{GHull,Ball}.

\begin{figure}[t]
\centering\includegraphics[width=0.8\textwidth]{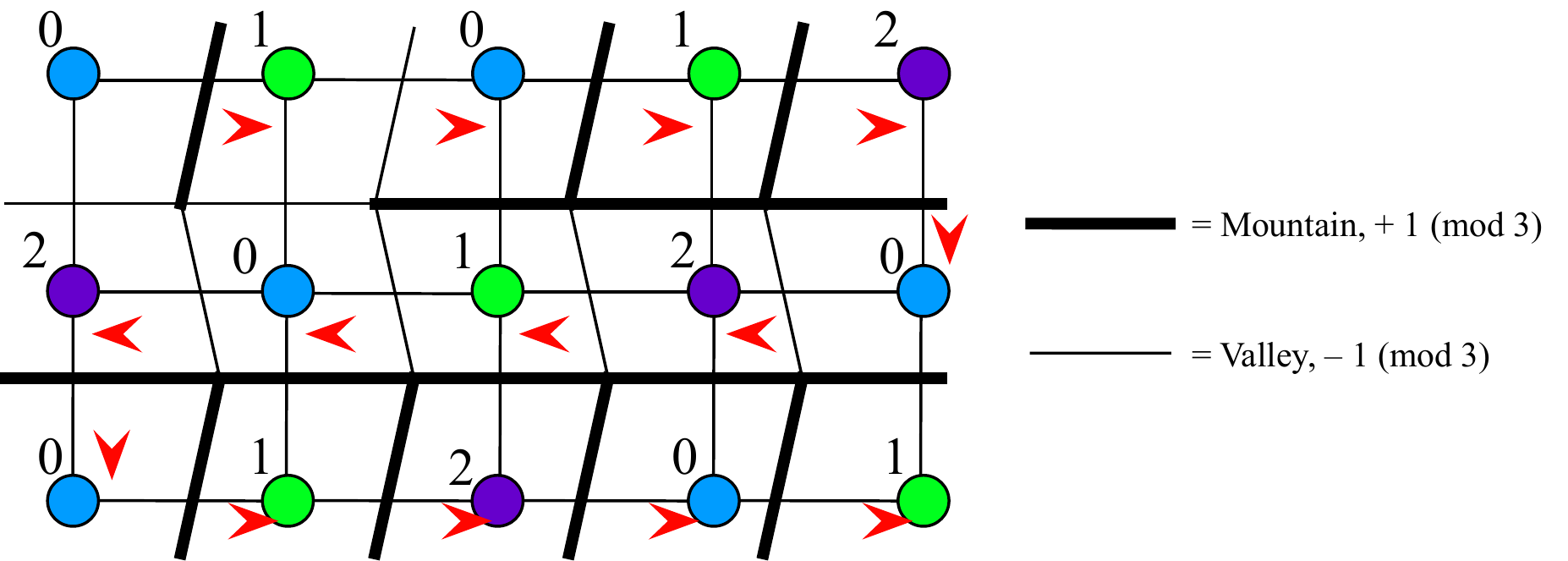}
\caption{Bijection between local flat foldings of the Miura-ori and 3-colorings of the squares of a grid}
\label{fig:Miurabijection}
\end{figure}

In order to use this, we need to know how face flips change the  3-coloring of the corresponding grid graph.

\begin{lemma}
\label{lem:miura-flip-3color}
Suppose we flip a face $F$ of a Miura-ori with a locally flat-foldable MV assignment.  Then the corresponding 3-coloring of the grid graph will change only at the vertex that corresponds to $F$.
Conversely, if we change the 3-coloring of a single vertex of the grid graph, the corresponding Miura-ori MV assignment will change by flipping a single face.
\end{lemma}

\begin{proof}
The correspondence between local flat-foldings and 3-colorings, with colors differing by $+1$ mod 3 across mountain folds and $-1$ mod 3 across valley folds, can be extended from the path described above to a directed version of the entire grid, as follows. Each grid edge that belongs to the path maintains its direction from the path itself. Each remaining grid edge crosses a parallel line of the folding pattern. The same parallel line is crossed by exactly one directed edge of the path, and we assign all of the other edges across that line the same direction.
A local case analysis, in which we consider the possible local flat foldings for the four edges surrounding each vertex of the folding pattern, shows that the correspondence between local flat-foldings and 3-colorings is consistent throughout the folding pattern.

Now, if we flip face $F$, the colors of the surrounding faces do not change, because they are all connected to each other in this directed grid graphs via paths that do not pass through the vertex corresponding to the flipped face.
Conversely, if we change a single grid graph vertex color, the same correspondence shows that the MV assignment cannot change except at the edges surrounding the corresponding face~$F$.
\end{proof}

With this equivalence between flips and vertex recolorings in hand, we can apply known techniques for grid colorings to obtain the corresponding results for Miura-ori face flips.

\begin{theorem}
Every two locally-valid MV assignments of the Miura-ori crease pattern can be converted to each other via face flips.
\end{theorem}

\begin{proof}
This follows from \autoref{lem:miura-flip-3color}
and from the already-known fact that every two 3-colorings of a grid can be converted to each other via single-vertex recolorings~\cite{Goldberg}.
More strongly, Goldberg et al.~\cite{Goldberg} prove that the number of recoloring steps needed, for an $m\times n$ grid with $m\le n$, is at most $2mn^2$. The same bound holds for face flips of the Miura-ori, where now $m$ and $n$ measure the number of parallelograms in the pattern in either direction.
\end{proof}

\begin{theorem}
Given any two locally-valid MV assignments of the Miura-ori crease pattern, it is possible to find
a minimum-length sequence of face flips that converts one to the other, in polynomial time.
\end{theorem}

\begin{proof}
The correspondence between MV assignments and grid 3-colorings, described above, can be carried out in polynomial time. We may then employ a height function on grid 3-colorings described by Luby et al.~\cite{Luby} to find a minimum-length sequence of recolorings that take the pair of colorings to a common third coloring.
This height function maps grid vertices to integers that are equal, modulo 3, to the value of the color on that vertex, and that differ by $\pm 1$ between adjacent vertices. There is a unique such function, up to a translation (the choice of the height for any one designated vertex), and it is easily calculated from the coloring. Whenever two colorings differ by a recoloring step, their height functions differ by adding or subtracting 2 from the height of a vertex.

Any sequence of recoloring steps that takes one coloring to another takes their height functions into each other (again, modulo a translation). The number of steps is therefore at least $1/2$ times the sum, over all grid vertices, of the differences in heights. There always exists a sequence of recoloring steps whose length is exactly this long, obtained by choosing a vertex whose height is a local minimum or maximum and that is not at the goal height and recoloring that vertex. This shortest recoloring sequence could be obtained in polynomial time, from the starting and ending height functions. However, to obtain it from the starting and ending grid colorings we must determine the unknown translation by which the starting and ending height functions differ. The number of choices of this translation is linear in the maximum side length of the grid, so we may try all possibilities in polynomial time.
\end{proof}

\section{Triangle lattice tessellations}\label{sec4}
%\section{Triangle lattice tessellations}\label{sec4}

%%%%%%%%%%%%%%%%%%%%%%%%%%%%%%%%%%%%%%%%
In this section we consider finite regions of the triangle lattice.
In Section~\ref{sec:triangle-reconfig} we show that the configuration space of locally-valid MV assignments is connected under face flips, with a linear diameter.  
In Section~\ref{sec:triangle-NP-hard} we show that it is NP-hard to find the minimum number of face flips to reconfigure between two locally-valid MV assignments.

We begin with some basic facts about locally-valid MV assignments in the triangle lattice.  
By Maekawa's Theorem, a mountain-valley assignment is locally valid if and only if 
for any vertex $v$,  either $v$ has 4 mountain folds and 2 valley folds, in which case we call it a \emph{mountain vertex}, or $v$ has 4 valley  folds and 2 mountain folds, in which case we call it a \emph{valley vertex}.  

Consider what happens at a valid vertex $v$ if we flip an incident face $f$.  If the two edges of $f$ incident to $v$ have opposite creases (one mountain and one valley),
then flipping $f$ preserves that property and $v$ remains valid---in fact, $v$ retains its mountain/valley designation
(Figure~\ref{fig:triangles}(a)).  
If the two edges of $f$ incident to $v$ are mountain creases and $v$ is a mountain vertex, then flipping $f$ changes $v$ to a valley vertex (Figure~\ref{fig:triangles}(b)).  Similarly, flipping two valley creases at a valley vertex creates a mountain vertex.
Finally, if the two edges of $f$ incident to $v$ are mountain creases and $v$ is a valley vertex (Figure~\ref{fig:triangles}(c)), or if the two edges are valley creases and $v$ is a mountain vertex (Figure~\ref{fig:triangles}(d)), then $f$ cannot be flipped, and we say that $v$ \emph{causes} $f$ to be not flippable.  To summarize:

\begin{claim}
A face can be flipped unless it has 2 mountain creases incident to a valley vertex, or 2 valley creases incident to a mountain vertex.
%A face $f$ in a locally-valid MV assignment is unflippable if and only if  $f$ has a vertex 
\label{lemma:bad-face}  
\end{claim}

In particular, note that a face with a mountain and a valley crease has only one vertex that can potentially cause it to be not flippable.

\begin{figure}[htb]
\centering\includegraphics[width=0.7\textwidth]{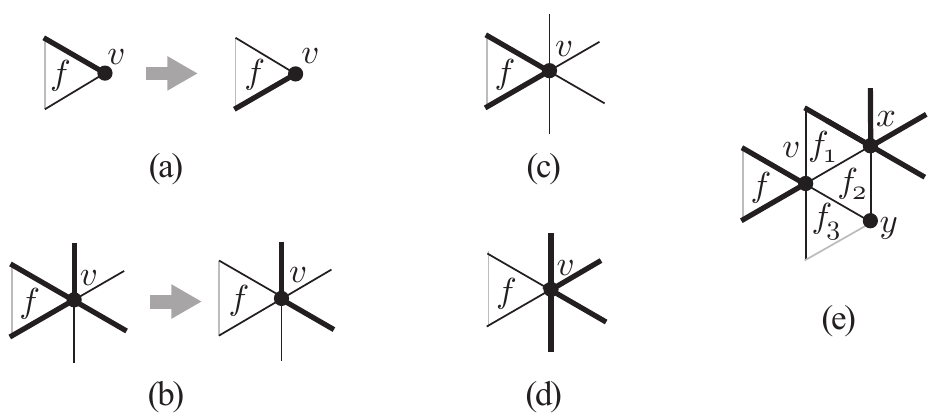}
\caption{(a--d) Some cases for flipping a triangle $f$ incident to vertex $v$: (a) one mountain crease (in bold) and one valley crease; (b) two mountain creases at a mountain vertex;  (c) valley vertex $v$ causes face $f$ to be not flippable; (d) mountain vertex $v$ causes face $f$ to be not flippable. (e) Illustration for Lemma~\ref{lemma:local-face-flip}.}
\label{fig:triangles}
\end{figure}

\begin{lemma} Suppose vertex $v$ causes face $f$ to be not flippable.  Let $f_1, f_2, f_3$ be the 3 faces incident to $v$ but not adjacent to $f$, where $f_2$ is the \emph{middle} face---the one opposite $f$.
Then at least one of $f_1, f_2, f_3$ is flippable.  Furthermore, if $f$ has both a mountain and a valley edge then after flipping one of $f_1, f_2, f_3$, face $f$ becomes flippable. 
\label{lemma:local-face-flip}
\end{lemma}
\begin{proof}  Suppose without loss of generality that $v$ is a valley vertex and the two edges of $f$ incident to $v$ are mountain creases.  
See Figure~\ref{fig:triangles}(e).
Note that $v$ cannot cause any of $f_1, f_2, f_3$ to be not flippable.   Suppose $f_2 = vxy$ is not flippable.  Then this is caused by one of its other vertices, say vertex $x$, and suppose that $x$ is incident to $f_1$ (otherwise relabel $f_1$ and $f_3$).  Since $xv$ is a valley, $xy$ must also be a valley, and $x$ must be a mountain vertex, so its other incident edges must be mountains.  Then  face $f_1$ has a mountain edge opposite $v$ and is flippable.  
\end{proof}

\subsection{Reconfiguring a triangle lattice}
\label{sec:triangle-reconfig}

To show that any locally-valid MV assignment of the triangle lattice can be reconfigured to any other using face flips, we use the standard technique of reconfiguring any locally-valid MV assignment to a ``canonical'' MV assignment.  
To reconfigure $A$ to $B$  via the canonical configuration $C$, we reconfigure $A$ to $C$, and then perform the reverse of the reconfiguration sequence that takes $B$ to $C$.  
Our canonical configuration $C$ has mountain folds on the $-30^\circ$ lines of the lattice, and valley folds on the other lines---see Figure~\ref{fig:triangle-reconfig}(a).

\begin{figure}[htb]
\centering\includegraphics[width=0.8\textwidth]{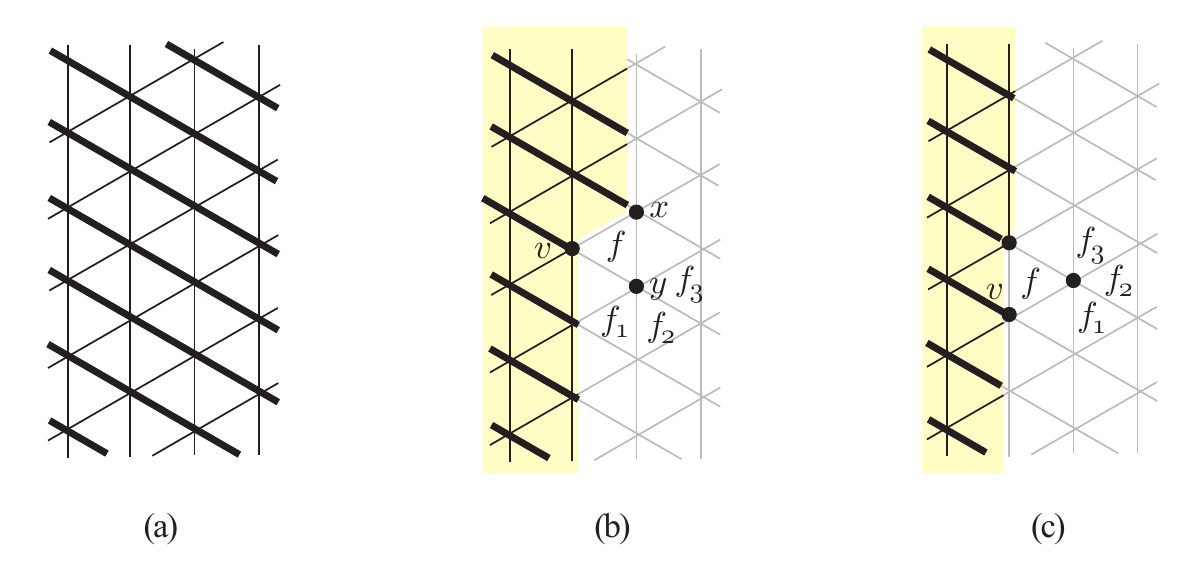}
\caption{(a) The canonical configuration $C$.  Mountain folds are in bold. (b) Fixing the diagonal edges down a column.  The completed region is coloured yellow. (c) Fixing the vertical edges down a column.}
\label{fig:triangle-reconfig}
\end{figure}

\begin{theorem}   
	\label{thm:triang-diameter}
	Any locally-valid MV assignment of the triangle lattice can be reconfigured to the canonical configuration $C$ using $2n$ flips, where $n$ is the number of faces.
\end{theorem}
\begin{proof}
We use an iterative process, proceeding from left to right, first fixing the diagonal edges down a column (Figure~\ref{fig:triangle-reconfig}(b)), and then fixing the vertical edges down a line (Figure~\ref{fig:triangle-reconfig}(c)), while always preserving the MV assignment in  the ``completed'' region (shown in yellow in the figures).  We address these two cases separately. 

\paragraph{Case 1.  Diagonal edges.} Suppose all the diagonal edges down a column match the canonical configuration up until the two edges incident with vertex $v$ on the left-hand side of the  column.  Let $v$'s neighbours across the column be $x$ and $y$, with $x$ above $y$.  Because vertex $v$ already has 3 incident valley creases from the completed region, $v$ must be a valley vertex and one of the edges $vx$, $vy$ must be a mountain and the other a valley.

If edge $vx$ is a valley then $vy$ is a mountain, and the two edges incident to $v$ are correct for the canonical configuration.
Thus, we may suppose that $vx$ is a mountain and $vy$ is a valley.  
Let $f$ be the face $vxy$.  Flipping $f$ corrects the two edges incident to $v$, so if $f$ is flippable, we are done.  
Thus, we may suppose that $f$ is not flippable.  By Lemma~\ref{lemma:bad-face}, $f$ must be unflippable because of vertex $x$ or $y$.  %We consider these cases:

We first show that $x$ cannot cause $f$ to be unflippable. Suppose it did.  
%\paragraph{Case 1a.  $f$ is not flippable because of vertex $x$.}  
Then both edges of $f$ incident to $x$ must be the same; furthermore, they must be mountains because $vx$ is a mountain.  Vertex $x$ already had one incident mountain crease from the completed region.  But then $x$ has at least 3 incident mountain edges so it must be a mountain vertex, and it cannot prevent $f$ from flipping. 

Next, suppose that $y$ causes $f$ to be unflippable.
%Then both edges of $f$ incident to $y$ must be the same; furthermore, they must be valleys because $vy$ is a valley.  Then $xy$ is also a valley, and vertex $y$ is a mountain vertex.
Let $f_1, f_2, f_3$ be the three faces incident to $y$ and not adjacent to $f$.
By Lemma~\ref{lemma:local-face-flip} at least one of $f_1, f_2, f_3$ is flippable.  Note that none of the edges of these faces are in the completed region.  (See  Figure~\ref{fig:triangle-reconfig}(b).) 
Thus, we can flip one of  $f_1, f_2, f_3$ without disturbing the completed region. 
Furthermore, $f$ will then be flippable by Lemma~\ref{lemma:local-face-flip}, since 
$f$ has a mountain and a valley crease.

\paragraph{Case 2.  Vertical edges.}
Suppose that all the vertical edges down a column match the canonical configuration up until the edge $uv$ with $u$ above $v$.  Let the triangle to the right of $uv$ be $f$, and suppose the $f$'s third vertex is $w$. 
If $uv$ is a valley, then it matches the canonical configuration.  So, suppose that $uv$ is a mountain.  If face $f$ is flippable, that would fix $uv$, so we may suppose that $f$ is not flippable.  We claim that neither $u$ nor $v$ can be the cause. Observe that $u$ and $v$ each have two incident mountain edges.  In order for $u$, say, to cause $f$ to be non-flippable, edge $uw$ must be a mountain like $uv$.  But then $u$ is a mountain vertex so it cannot prevent $f$ from flipping.  The same argument applies to $v$.

Thus, $f$ must be non-flippable because of vertex $w$.  
Let $f_1, f_2, f_3$ be the three faces incident to $w$ and not adjacent to $f$.
By Lemma~\ref{lemma:local-face-flip} at least one of $f_1, f_2, f_3$ is flippable.  Note that none of the edges of these faces are in the completed region.  (See  Figure~\ref{fig:triangle-reconfig}(c).) 
Thus, we can flip one of  $f_1, f_2, f_3$ without disturbing the completed region.
Furthermore, $f$ will then be flippable by Lemma~\ref{lemma:local-face-flip},  since $f$ has a mountain and a valley crease.
\end{proof}

\subsection{Finding the minimum number of face flips is NP-hard}
\label{sec:triangle-NP-hard}

We show that the problem of finding the sequence of face flips of minimum length between two given crease patterns is NP-hard. 
For hardness, we reduce from \textsc{$k$-Vertex-Cover} in max-degree-3 graphs which is NP-complete~\cite{lichtenstein1982planar}.
Given graph $G$ the problem asks for the a subset $S$ of $V(G)$ so that every edge in $E(G)$ is incident to at least one vertex in $S$, and $|S|\le k$.
%For the hardness proof, we consider the decision version of the problem that asks for a vertex cover of at most at least $k$.
We assume that our input is a max-degree-3 graph $G$ embedded in the hexagonal grid.
Notice that this does not mean that $G$ is a hexagonal grid graph since the grid is bipartite and admits a polynomial solution for the \textsc{$k$-Vertex-Cover} problem.
Rather, edges in $G$ can be drawn as a path in the grid admitting bends.
Such embeddings of polynomial size can be computed in polynomial time~\cite{hex}.

\begin{figure}[h]
	\centering
	\includegraphics[width=\linewidth]{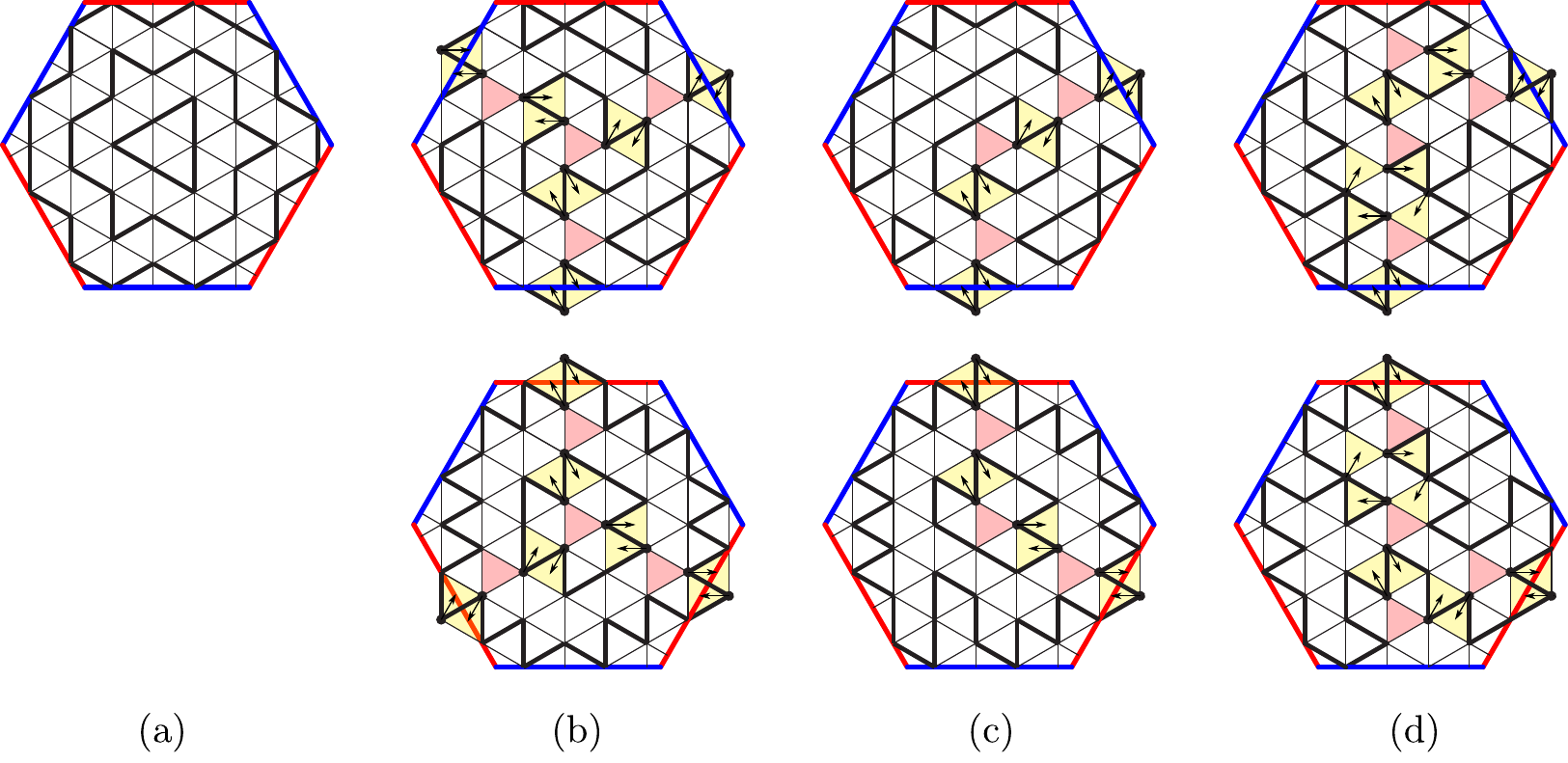}
	\caption{Reduction from \textsc{$k$-Vertex-Cover}. (a) Filler, (b) deg-3, (c) deg-2, and (d) bend gadgets. All yellow faces are unflippable and the arrows indicate the vertex that makes it unflippable.}
	\label{fig:gadgets}
\end{figure}

\begin{theorem}
	Given two MV assignments of the triangle lattice $A$ and $B$, it is NP-complete to decide whether A can be reconfigured into B using at most $m$ face flips.
\end{theorem}

\begin{proof}
	The membership in NP is a consequence of Theorem~\ref{thm:triang-diameter}.
	We proceed with the reduction from \textsc{$k$-Vertex-Cover}.
	Given a drawing of a graph $G$ in the hexagonal grid, we construct a MV assignment $A$ as follows.
	We will use the gadgets shown in Figure~\ref{fig:gadgets}.
	We first obtain a triangular grid by inserting one \emph{auxiliary} vertex in each hexagon and connecting it to each vertex of the hexagon.
	We use the triangle grid to tile the plane using hexagons so that each hexagon corresponds to either a vertex of the original hexagonal grid or an auxiliary vertex.
	If a tile corresponds to a vertex that is not used in the embedding (a degree-3 vertex), we replace it with a \emph{filler} gadget shown in Figure~\ref{fig:gadgets}~(a) (\emph{deg-3} gadget shown in Figure~\ref{fig:gadgets}~(b)).
	If the tile corresponds to a degree-2 vertex (bend), we replace it with the \emph{deg-2} gadget shown in Figure~\ref{fig:gadgets}~(c) (\emph{bend} gadget shown in Figure~\ref{fig:gadgets}~(b)) or a $120^\circ$ rotation.
	This defines the MV assignment $A$.
	$B$ is obtained from $A$ by flipping the assignment of the perimeter of each yellow diamond or triangle as shown in Figure~\ref{fig:hardness-change}.
	All vertices in the interior of gadgets are locally valid.
	By construction a blue edge of a gadget will be matched with a red edge. 
	Every vertex on the boundary of a gadget has either zero or an even number of incident mountains and at least one incident valley. 
	A vertex on a red edge is always incident to a mountain crease.
	Then, the MV assignment $A$ and $B$ are locally valid.
	We now determine $m$.
	Let $b$ be the number of bends in the drawing of $G$.
	We set $m=2 k+8 |E(G)|+14 b$.
	
	\begin{figure}[h]
		\centering
		\includegraphics[width=0.2\linewidth]{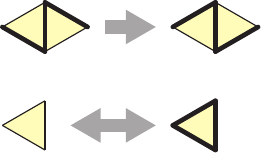}
		\caption{Changes in the MV assignment.}
		\label{fig:hardness-change}
	\end{figure}
	
	($\Rightarrow$) Assume that the \textsc{$k$-Vertex-Cover} instance admits a solution $S$.
	We will construct a sequence of at most $m$ flips that brings $A$ into $B$.
	Note that all pink triangles are flippable and, because no two are adjacent, flipping any subset of them will not render an unflipped pink traingle not flippable.
	We call the pink central triangle of the deg-3 and deg-2 gadgets \emph{vertex triangles}.
	If a vertex is in $S$, flip the vertex triangle of its corresponding gadget.
	Every edge corresponds to a chain of an even number of pink triangles between the vertex triangles of its corresponding endpoints.
	Since every edge $e$ has at least an endpoint in $S$, we can flip alternating pink triangles ($1+2b_e$ of them where $b_e$ is the number of bends of $e$) so that every yellow diamond along $e$ has an adjacent flipped triangle.
	Let $P$ be the set of faces flipped so far.
	Note that every face colored yellow is initially unflippable, but at the current state at least one face in each diamond or one face in each group of 4 yellow triangles in a bend gadget is flippable.
	By flipping such face an adjacent yellow face becomes flippable.
	Proceed by flipping all yellow faces.
	Note that a vertex of a face in $P$ is incident to exactly two yellow faces.
	Then, every face in $P$ is flippable at the current state.
	Now, we can flip all faces in $P$ to obtain $B$.
	The total number of flips is $2 |S|+8 |E(G)|+14 b<m$.
	
	($\Leftarrow$) Assume that there exist a sequence $F$ of flips transforming $A$ into $B$.
	Notice that it would suffice to flip every yellow face, however, they are unflippable in $A$.
	We show that yellow faces are flipped exactly once and we can assume that pink faces are either not flipped or flipped twice.
	Assume that $F$ does not flip a yellow face $f$. 
	Then, by construction of $B$, there are at least two white or pink faces that share an edge with $f$ that must be flipped an odd number of times.
	However, flipping those faces create assignments that do not match $B$, hence requiring further flips.
	By propagating this argument, we conclude that every white and pink face must be flipped an odd number of times while yellow faces are either never flipped or flipped an even number of times.
	The total number of flips is clearly greater then $m$ if $k<|V(G)|$, a contradiction.
	Therefore, yellow faces are flipped an odd number of times.
	Since the yellow faces are not flippable in $A$, at least one face adjacent to each yellow component must be flipped an odd number of times. 
	If there is a white face $f_w$ flipped an odd number of times is adjacent to a pink face $f_p$ that is not flipped, we can change the flipping sequence so that $f_p$ is flipped instead of $f_w$ because the same yellow faces (or a superset) become flippable in both sequences.
	We can assume that no other faces are flipped since they must be flipped an even number of times and they do not affect whether yellow faces become flippable or not.
	Except for vertex triangles, flipping a pink face can make a yellow face in at most two different yellow components become flippable. 
	The number of yellow components between vertex triangles corresponding to an edge $e$ is odd and the total number of yellow faces in such components is $6+10b_e$ where $b_e$ is the number of bends od $e$. 
	Then, excluding flips of vertex triangles, at least $8+14 b_e$ are necessary for each edge $e$.
	If an edge achieves such lower bound, then at least one of its endpoints is a vertex triangle that was flipped an even number of times.
	If more than $8+14 b_e$ flips were performed for a given edge $e$, then there is a yellow component adjacent to two pink faces that are flipped at least twice each.
	We can modify the sequence while not increasing the number of flips so that the only such yellow components are the ones adjacent to vertex triangles.
	In such a solution, at most $k$ vertex faces can be flipped by the definition of $m$.
	Then, we can obtain a solution for the \textsc{$k$-Vertex-Cover} instance by selecting the vertices corresponding to the vertex faces that are flipped.
\end{proof}

\section{Conclusion}

We have seen how face flips in flat-foldable origami tessellations describe different kinds of inherent structure between locally-valid MV assignments.  We employed a variety of techniques, such as modifications of the dual graph, height functions, and reconfiguration schemes to prove connectivity of the face flip configuration space (the origami flip graph) and to find minimal sets of faces to flip from one locally-valid MV assignment to another.  The latter was shown to be NP-hard in the case of triangle lattice crease patterns.  It is interesting to see how different origami tessellations require such  different techniques to analyze their origami flip graph structure.
 This indicates that face flip configuration space graphs are rich in structure among all flat-foldable crease patterns. 

 These results relate to studies in materials science and mechanical engineering on applied origami (such as \cite{Assis,Evans,Silverberg}). Face flips provide a way of studying the likelihood of a material either being manipulated from one MV assignment to another, or folding to a state that is ``close" to the target MV assignment state, where ``close" could be interpreted as two vertices close to each other in the origami flip graph.  Further, if the origami flip graph is disconnected, then its different components could identify folded states that have very little chance of being achieved by an actual folded material with a given target MV assignment.

Many questions can be posed for further work.  For example: Is the NP-hardness of finding a minimal path in the origami flip graph in the triangle lattice  due to the fact that the vertices in the crease pattern have degree 6 (aside from the boundary vertices), as opposed to the tessellations in Sections~\ref{sec2} and \ref{sec3} whose vertices have degree 4?  Are other origami tessellations in the engineering literature, such as \cite{Evans}, amenable to the face flip techniques presented here?  

\section*{Acknowledgments} This work was conducted at the 2018 Bellairs Workshop on Computational Geometry, co-organized by Erik Demaine and Godfried Toussaint.  We thank the other participants of the workshop for helpful discussions.  We also thank Sarah Nash and Natasha Ter-Saakov for helpful comments on an earlier draft of this work.
H. A. A. was partially supported by NSF grants CCF-1422311 and CCF-1423615.
D. E. was partially supported by NSF grants  CCF-1618301 and CCF-1616248.  T. C. H. was supported by NSF grant DMS-1906202.

\bibliographystyle{plainurl}
\raggedright
\bibliography{FaceFlips}
\end{document}